\documentclass[12pt]{article}

\usepackage[T2A]{fontenc}
\usepackage[margin=1in]{geometry} 
\usepackage{amsmath,amsthm,amssymb}
\usepackage[utf8]{inputenc}
\usepackage[english]{babel}
\usepackage{graphicx}
\usepackage{tablefootnote}
\usepackage{float}
\usepackage[ruled, vlined, linesnumbered, longend]{algorithm2e}

\usepackage{epstopdf}

\usepackage{amsfonts, amsmath, amsthm}
\usepackage{amssymb}
\usepackage{url}
\usepackage{hyperref}

\usepackage{indentfirst}

\usepackage{mathtools}
\usepackage{makecell}
\usepackage{tikz}
\usepackage{dsfont}
\usepackage{secdot}
\usepackage{graphicx}

\usepackage{float}
\usepackage{comment}

\usepackage{multicol}
\usepackage{lipsum}
\usepackage{mwe}

\usepackage{orcidlink}

\usepackage{caption}
\usepackage{subcaption}

\sectiondot{section}
\sectiondot{subsection}
\theoremstyle{definition}

\newtheorem{theorem}{Theorem}
\newtheorem{lemma}{Lemma}

\newtheorem{definition}{Definition}

\newtheorem{statement}{Proposition}

\makeatletter

\allowdisplaybreaks

\begin{document}

\begin{center}

\textbf{\large{Reducing the upper bound for the Borsuk number in $\mathbb{R}^4$ to 8}}
\\[1em]
{ 
Alexander Tolmachev\, \orcidlink{0009-0006-7711-3005}$^{1,2,3,*}$, 
Vsevolod Voronov\, \orcidlink{0000-0003-3835-6144}$^{2, 1,**}$
}
\\[1em]
\footnotesize{
\textit{$^1$Moscow Institute of Physics and Technology, Moscow, Russia} \\
\textit{$^2$Caucasus Mathematical Center of Adyghe State University, Maikop, Russia} \\
\textit{$^3$Skolkovo Institute of Science and Technology, Moscow, Russia}

$^{*}$tolmachev.ad@phystech.edu $^{**}$v-vor@yandex.ru
}

\begin{abstract}
The Borsuk number $b(n)$ of $n$-dimensional Euclidean space $\mathbb{R}^n$ is the smallest integer such that any set $F \subset \mathbb{R}^n$ of unit diameter can be partitioned into $b(n)$ subsets of strictly smaller diameter. For $n=4$, the best known upper bound $b(4) \leq 9$ follows from a construction by M. Lassak (1982). In the present paper, we construct partitions of several variants of the truncated Lassak cover into 8 parts of diameter less than 1, thereby showing that $b(4) \leq 8$.

\end{abstract}

\end{center}

\section{Introduction} 

In 1933, K. Borsuk~\cite{borsuk} posed the following question: \textit{Can every bounded set with positive diameter in $\mathbb{R}^n$ be partitioned into $n + 1$ subsets of smaller diameter?} This problem has become one of the most famous in combinatorial geometry~\cite{borsuk1, borsuk2, borsuk3}. It is usually called the Borsuk problem or the Borsuk conjecture. The first counterexample for $n=2015$ was found by J.~Kahn and G.~Kalai in~\cite{KahnKalai1993}. Later counterexamples were found in smaller dimensions (see \cite{Raigorodskii2007} for a detailed review). At the moment, finite sets that cannot be partitioned into $n+1$ parts of smaller diameter have been constructed for $n \geq 64$~\cite{bondarenko2014borsuk,jenrich201464}. Since it is known that the answer to the original Borsuk question is negative in general, it is of interest to determine the smallest number of parts, depending on the dimension, for which such a partition is possible.

\begin{definition} Let $b(n)$ be the minimum number of parts such
that any bounded subset $A \subset \mathbb{R}^n$ can be divided into $b(n)$ parts of smaller diameter. The quantity $b(n)$ is called the Borsuk number.
\end{definition}

The asymptotic lower and upper bounds~\cite{raigorodskii1999bound,bourgain1991convering} have the form

\[
    (1.2255...+o(1))^{\sqrt{n}} < b(n) < (1.224...+o(1))^{n}.
\]
However, for these estimates to work well, the dimension must be sufficiently large. 

The lower bound $b(n) \ge n + 1$ follows from the fact that the regular $n$-dimensional simplex cannot be partitioned into $n$ parts of smaller diameter. 
Observe that a unit segment can be divided into two segments of smaller length. Thus, $b(1)=2$. K.~Borsuk observed that $b(2)=3$. The exact value of $b(3) = 4$ was by J. Perkal~\cite{perkal1947} in 1947 and independently, by H. G. Eggleston~\cite{eggleston1955} in 1955. Later  B.~Grunbaum~\cite{grunbaum1957simple} and A.~Heppes~\cite{Heppes1957} proposed  simplified proofs. The estimate for the maximum diameter of the parts has been improved from $0.9887\ldots$ (B.~Grünbaum, 1957~\cite{grunbaum1957simple}) to $0.98$ (V.V.~Makeev and L.~Evdokimov, 1997~\cite{makeev2000}). In our recent work, this quantitative estimate was further improved.

\begin{theorem}[~\cite{DAM_2022_Tolmachev}]
        Any set of unit diameter in $\mathbb{R}^3$ can be partitioned into 4 parts of diameter less than $0.966$.
\end{theorem}

Currently it is known that the Borsuk conjecture holds for $n \le 3$ and does not hold for $n \ge 64$. For $n \ge 4$, the best constructive upper bound for the Borsuk number in small dimensions is $b(n) \le 2^{n - 1} + 1$~\cite{lassak1982}, resulting in $b(4) \le 9$. The present paper proposes a computational approach to improving upper bound for Borsuk number in $\mathbb{R}^4$: 

\begin{theorem}
    Any set of unit diameter in $\mathbb{R}^4$ can be partitioned into 8 parts of diameter less than $0.9999$.
    \label{main_theorem}
\end{theorem}

Consequently, $b(4) \leq 8$. 

Another algorithm for the numerical search for partitions and the search for counterexamples was proposed in~\cite{zong2}. Quantitative estimates in the Borsuk problem have also been studied beyond $\mathbb{R}^n$, in finite-dimensional Banach spaces~\cite{ZhangMengWu2023, QiZhangLyuWu2025}.

Using the techniques employed in these works, one can compute bounds $p_0$, $p_1$ with $1 < p_0 < 2 < p_1$ such that the statement of Theorem~\ref{main_theorem} remains valid for $\mathbb{R}^4$ equipped with the metric defined by the $p$-norm.

\section{Universal covers and covering systems}

The computational approach to the estimates of the Borsuk number~\cite{kupavskii_d_3_5,Filimonov,DM_2021_Tolmachev, DAM_2022_Tolmachev} is based on the following definitions:

\begin{definition} A set $C$ is called a universal covering set or a \textit{universal cover} if every set $F \subset \mathbb{R}^n$ of unit diameter can be  covered by an isometric copy of $C$. 
\end{definition}

\begin{definition} A collection of sets $\mathcal{S} = \{C_{\alpha} \}_{\alpha \in I}$ is called a \textit{universal covering system (UCS)} if every set $F  \subset \mathbb{R}^n$ of unit diameter can be  covered by an isometric copy of some set $C \in \mathcal{S}$. Here $I$ is a (possibly infinite) set of indices.
\end{definition}

\begin{theorem}[H.~Jung~\cite{Jung1901}, 1901]\label{{thm:jung}} The ball $B \subset \mathbb{R}^n$ with radius $r_n = \sqrt{\frac{n}{2n+2}}$ is a universal cover in $\mathbb{R}^n$.
\end{theorem}

 This ball is called the \textit{Jung ball}. In order to prove the constructive upper bound for the Borsuk number, M. Lassak  considered the following universal cover:

\begin{theorem}[M.~Lassak~\cite{lassak1982}, 1982]\label{thm:lassak}
    Let $c_0, c_1 \in \mathbb{R}^n$ be arbitrary points at a Euclidean distance of $r = \sqrt{\frac{n}{2n+2}}$. The set $L_{(n)} = B_0 \cap B_1 \subset\mathbb{R}^n$ is a universal cover in $\mathbb{R}^n$, where 
    $$B_0 = \{x \in \mathbb{R}^n: \|x - c_0\|_2 \le r \}, \quad B_1 = \{x \in \mathbb{R}^n: \|x - c_1\|_2 \le 1 \}.$$
\end{theorem} 

Now let us introduce the well-known method that can be used to construct a UCS from a single universal cover.

\begin{statement}[\cite{grunbaum1957simple}]
\label{st:cutting}
Let $\mathcal{S} = \{C_{\alpha} \}_{\alpha \in I}$ be a UCS. Let $h^+, h^-$ be a pair of parallel hyperplanes at a unit distance apart, and $H^+$, $H^-$ be intersecting half-spaces bounded by $h^+$, $h^-$ respectively. Suppose that there exists $C \in \mathcal{S}$ such that $C \setminus H^+ \neq \emptyset$ and $C \setminus H^- \neq \emptyset$.

Then the collection
   \[
        \mathcal{S}' = \mathcal{S} \setminus \{C\} \cup \{C \cap H^+, C \cap H^-\}  
    \]

is a universal covering system.
\end{statement}

The process of cutting the covers allows one to make the maximum diameter of $C_\alpha$ arbitrarily close to unity. But at the same time, the number of elements of the covering system grows exponentially in relation to the number of operations applied to each cover. In this paper, we limit ourselves to just a few truncations and a covering system $\mathcal{S}$ consisting of 4 sets.

In the next sections, we create partitions of each set $C_i$, $i=1,2,3,4$, into 8 parts of diameter less than 1. This guarantees that each set $A \subset \mathbb{R}^4$ can be covered by one of the elements of $\mathcal{S}$, so $A$ can be divided into 8 parts of diameter less than 1. The verification procedure for the diameter computation finalizes the proof of Theorem~\ref{main_theorem}.

\section{Truncation of the Lassak cover}

In this section, we show that any body $F \subset \mathbb{R}^4$ with $\operatorname{diam} F = 1$ can be covered by at least one of the four sets obtained by truncating (the projection of) the Lassak cover $L^{(4)}$ using eight faces of a rhombic dodecahedron with a given center $x \in F$ and one additional hyperplane. The truncated rhombic dodecahedron was shown by V.V.~Makeev~\cite{makeev2000} to be a universal cover in $\mathbb{R}^3$. Since we require a cover with a fixed center, we cannot use all 12 faces and 3 truncations as in V.V.~Makeev's construction.

The intersection of the two spheres forming the boundary of the Lassak cover is the boundary of the Jung ball in dimension $n-1$.
In particular, for $n=4$, this intersection is a sphere $S'$ of radius $r_3 = \sqrt{\frac{3}{8}}$. We choose coordinates such that $S' \subset p_0 = \{x \in \mathbb{R}^4 : x_4 = 0\}$, and the centers of the spheres lie on the line $x_1 = x_2 = x_3 = 0$. 

Let $A \subset \mathbb{R}^4$ be a convex set with $\operatorname{diam} A = 1$. Let $A'$ be the projection of $A$ onto the hyperplane $p_0$. Observe that $\operatorname{diam} A' \leq \operatorname{diam} A = 1$.

\begin{statement}\label{st:sym_center}
    A closed convex set $A \subset \mathbb{R}^n$ can be covered by the Lassak cover $L_{(n)}$ in such a way that the symmetry axis of $L_{(n)}$ intersects $A$.
\end{statement}

\begin{proof}
    By construction of the Lassak cover (see Theorem~\ref{thm:lassak}), it is assumed that $c_1 \in A$. The symmetry axis of $L_{(n)}$ is the line passing through $c_0$ and $c_1$.
\end{proof}

Define the halfspaces with the boundary at $\frac{1}{2}$ from the origin as 
\[
    H_i^+ = H^+(u_i) = \left\{x \in \mathbb{R}^4 \mid \langle x, u_i \rangle \leq \frac{1}{2} \right\}, \quad H_i^-  = H^-(u_i) = \left\{x \in \mathbb{R}^4 \mid \langle x, u_i \rangle \geq -\frac{1}{2} \right\}.
\]

Let  $H_i = H_i^+ \cap H_i^-$ be the ``layer'' between a pair of parallel planes at unit distance.

To keep the notation to a minimum  in the rest of the section, we consider covering a three-dimensional body by \emph{four-dimensional ``tubes''}  whose projections onto $p_0$ form a UCS in three-dimensional space.

\begin{statement}
    \label{st:brick}
    For a given three-dimensional body $A' \subset p_0$ containing the origin, there exist unit vectors $u_1, u_2 \in p_0$ such that $u_1$ and $u_2$ are orthogonal and $A'$ is contained in the set  

    \[
    H_1 \cap H_2 = \left\{ x \in \mathbb{R}^4 \mid  - \frac{1}{2} \leq  \langle u_1, x \rangle \leq  \frac{1}{2}, \;  - \frac{1}{2} \leq \langle u_2, x \rangle  \leq \frac{1}{2} \right\}. 
    \]
\end{statement}

\begin{proof}
    Consider an arbitrary hyperplane $p'$ that does not coincide with $p_0$ and passes through $0$. Let $\omega$ be the unit circle centered at $0$ lying in the two-dimensional plane $p' \cap p_0$. Consider the vector $v(t)$ traversing $\omega$ as $t \in [0, 2\pi]$. Let the hyperplane $q(v, \alpha)$ be given by the equation $\langle v, x \rangle = \alpha$. Define  

    \[
        f: \omega \to \mathbb{R}_{>0}, \quad f(v) = \inf \{ \alpha >0 \mid q(v, \alpha) \cap \operatorname{int}  A' = \emptyset \}.
    \]

    The function $f(v)$ is continuous on the circle $\omega$. Moreover,  

    \[
        f(v)+f(-v)\leq \operatorname{diam} A' \leq  1.
    \] 

    Fix $v_0 \in \omega$. If $f(v_0) \leq \frac{1}{2}$ and $f(-v_0) \leq \frac{1}{2}$, then we can set $u_1 = v_0$. Then $A'$ lies between the hyperplanes $\langle u, x \rangle = \pm \frac{1}{2}$.

Otherwise, either $f(v_0) < \frac{1}{2}$, $f(-v_0) > \frac{1}{2}$, or  $f(v_0) > \frac{1}{2}$, $f(-v_0) < \frac{1}{2}$. Consider $f(v)$ as $v$ moves along the semicircle from $v_0$ to $-v_0$. By the intermediate value theorem, there exists $v^*$ such that $f(v^*) = \frac{1}{2}$; consequently, $f(-v^*) \leq \frac{1}{2}$, and we can set $u_1 = v^*$.

Next, we carry out the same reasoning for the hyperplane $p''$ orthogonal to  $u_1$. In the two-dimensional plane $p'' \cap p_0$, we find a vector $u_2$ satisfying the required condition. 
\end{proof}

Note that here we did not need the more complicated topological arguments from~\cite{makeev2000}.

\begin{figure}
    \centering
    \includegraphics[width=0.5\linewidth]{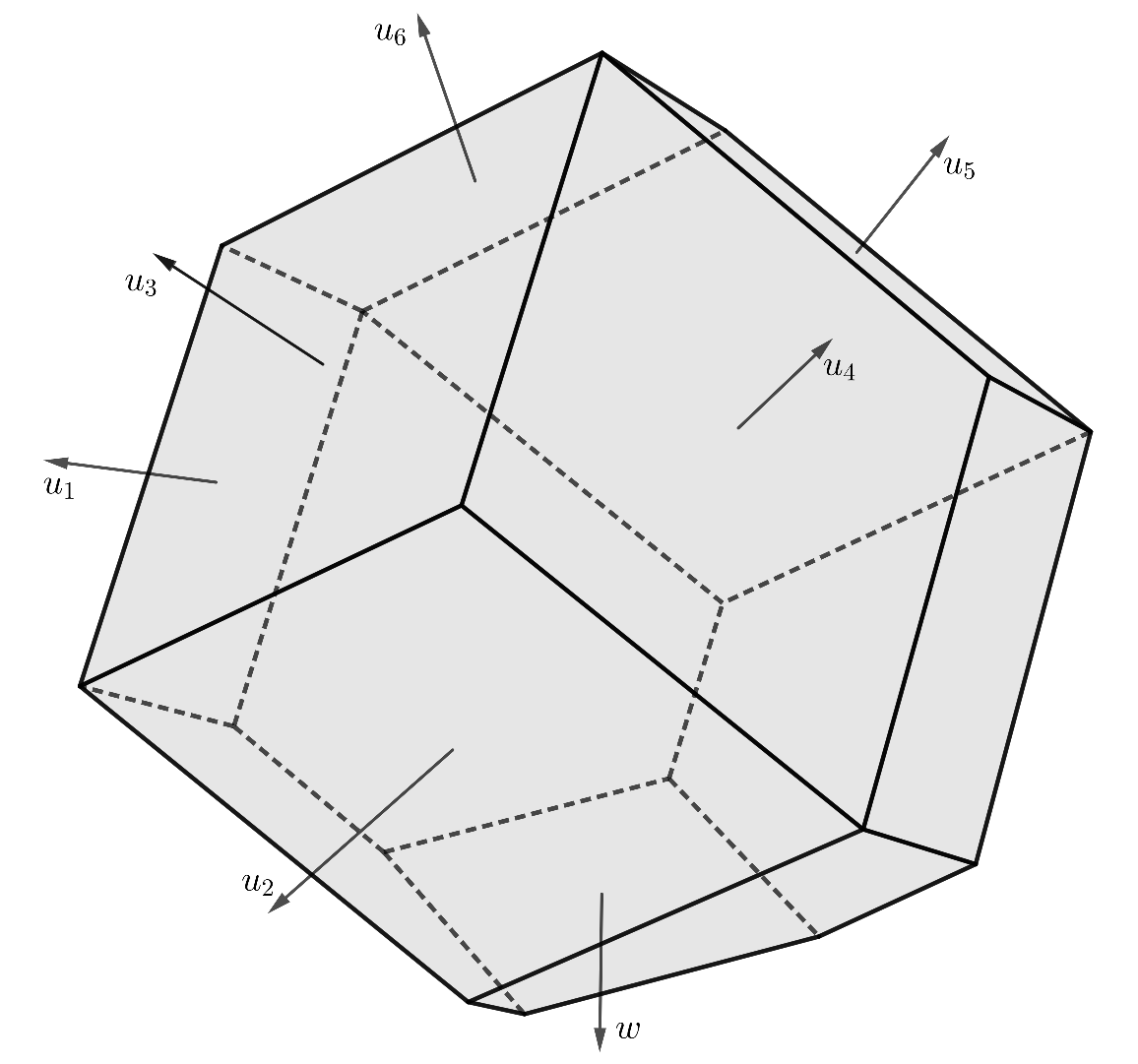} 
    \caption{Truncated rhombic dodecahedron.}
    \label{fig:rd_normals}
\end{figure}

\begin{statement}
\label{prop:ucs3d}
    Let $u_1, \dots, u_6 \in p_0$ be the unit normals to the six pairs of faces of a rhombic dodecahedron, and let the unit vector $w$ be orthogonal to $u_1$ and $u_2$ (see Fig.~\ref{fig:rd_normals}). 
    
    Then the four sets
    \[
        U_1 = H_1 \cap H_2 \cap H_3^+ \cap H_4^+ \cap H_5^+ \cap H_6^+,
    \]
    \[
        U_2 = H_1 \cap H_2 \cap H_3^+ \cap H_4^+ \cap H_5^+ \cap H_6^-,
    \]
    \[
        U_3 = H_1 \cap H_2 \cap H_3^+ \cap H_4^+ \cap H_5^- \cap H_6^-,
    \]
    \[
        U'_4 = H_1 \cap H_2 \cap H_3^+ \cap H_4^- \cap H_5^+ \cap H_6^- \cap H^+(w),
    \]

    form a universal covering system in the hyperplane $p_0$. Moreover, for any set $A'$ of unit diameter containing the origin, there exists a rotation $\phi$ of $p_0$ such that $\phi A'$ is covered by at least one of the sets $U_i$.
\end{statement}

\begin{proof}
    Fix some set $A$ of unit diameter. By Proposition~\ref{st:sym_center}, one can assume that its projection $A'$ contains the origin.  By Proposition~\ref{st:brick}, we choose a pair of orthogonal vectors $u_1, u_2$ such that $A$ is contained in $H_1 \cap H_2$. We complete the system of normals to the faces of the rhombic dodecahedron $u_1, \dots, u_6$. Then by Proposition~\ref{st:cutting}, $A$ is contained in at least one of the 16 sets of the form
     \[
        H_1 \cap H_2 \cap H_3^\pm \cap H_4^\pm \cap H_5^\pm \cap H_6^\pm 
    \]
    for some choice of signs. Taking into account the symmetries of the rhombic dodecahedron, we obtain four sets $U_i$, $i=1,2,3,4$.  $U_1, U_2, U_3$ are defined above; 
    \[
     U_4 = H_1 \cap H_2 \cap H_3^+ \cap H_4^- \cap H_5^+ \cap H_6^-.
    \]
    
    Moreover, the sets
 \[
        U'_4 = U_4 \cap H^+(w), \quad
        U''_4 = U_4 \cap H^-(w),
    \]
    are congruent, so it suffices to consider one of them.
\end{proof}

Observe that we could also have used the additional truncation of $U_3$ without increasing the number of elements of the UCS, but since a partition of $U_3$ was found, the truncation was not needed.

Now it remains to observe that for any set $A \subset \mathbb{R}^4$ covered by the Lassak cover $L_{(4)}$, we can choose one of the polyhedra covering its projection onto $p_0$. Consequently, in $\mathbb{R}^4$ there exists a UCS consisting of four variants of the truncation of the Lassak cover.

\begin{statement}\label{st:ucsR4}
The collection of four sets $\mathcal{S} = \{L_{(4)} \cap U_1, L_{(4)} \cap U_2, L_{(4)} \cap U_3, L_{(4)} \cap U'_4\}$ is a UCS in $\mathbb{R}^4$ 
\end{statement}

\begin{proof}
    Let $A_0 \subset \mathbb{R}^4$ be an arbitrary set of unit diameter. Consider an isometric copy $A \subset L_{(4)}$ of $A_0$. Let $A'$ be the projection of $A$ onto $p_0$. By Proposition~\ref{prop:ucs3d}, there exists a rotation  $\psi$ about the axis $0x_1$   such that $\psi A' \subset p_0$ is covered by the set $U \in \{U_1, U_2, U_3, U'_4\}$. By the construction of the Lassak cover, $\psi L_{(4)} = L_{(4)}$. Therefore, $\psi A \subset L_{(4)} \cap U$, as required. 
\end{proof}

\section{Computational methods}

In this section we describe the constructive procedure used to partition the truncated Lassak cover according to a fixed combinatorial scheme together with a geometric realization determined by a system of directions and a choice of a center point. Throughout the section we use the notation introduced earlier: $L_{(4)}^H = L^H$ denotes the Lassak cover truncated by a family of hyperplanes $H = \{h_1, h_2, \dots, h_m\}$, and $c_0$ denotes the center of the ball $B_0$ in the representation $L_{(4)} = B_0 \cap B_1$.  
Here we assume that $L^H$ is an element of the universal covering system $\mathcal{S}$ introduced above.

\subsection{Directions induced by the \texorpdfstring{$4$}{4}-dimensional hypercube}

The combinatorial structure of the partition is derived from the $4$-dimensional hypercube. Consider the set of its vertices $V = \{-1,1\}^4 \subset \mathbb{R}^4$. For each vertex $v \in V$ we define the normalized direction $d(v) = \frac{v}{\|v\|}$. Let $D = \{ d_1,\dots,d_{16} \}$
be the set of all such directions. These directions determine rays starting from a common center and intersecting the boundary of $L^H$. 

Let $\mathcal{F} = \{F_1,\dots,F_8\}$ be the family of index sets corresponding to these facets, where for $j \in \{1, 2, 3, 4\}$:
$$F_{2j - 1} =\{i: v_i = (v_{i, 1}, v_{i, 2}, v_{i, 3}, v_{i, 4}) \in V: v_{i, j} = -1 \},$$
$$F_{2j} =\{i: v_i = (v_{i, 1}, v_{i, 2}, v_{i, 3}, v_{i, 4}) \in V: v_{i, j} = 1 \}.$$ Thus each $F_i$ selects eight directions among the sixteen rays.

\subsection{Fixing the combinatorial structure}

Let $w$ be a point inside $L^H$ that serves as the apex of the partition. The construction is initialized with $w = w_0$. During the optimization stage described below, the point $w$ is allowed to move within the interior of $L^H$. The following property ensures that the construction remains well-defined and corresponds to some partition of $L^H$.

\begin{lemma}
If $w \in \operatorname{int}(L^H)$, then every ray $w + \lambda d_i$, where $\lambda \ge 0$, with $d_i \in D$, intersects the boundary $\partial L^H$ at a unique point $x_i$, where $i \in \{1, 2, \dots, 16\}$. Additionally, $w \in \operatorname{ConvexHull}\left(\{x_1, \dots, x_{16}\}\right)$.
\end{lemma}

\begin{proof}
The set $L^H$ is convex, being an intersection of convex sets (balls and halfspaces). Since $T$ lies in its interior, the ray intersects the boundary exactly once in the forward direction. Because the truncated Lassak cover $L^H$ is a convex set containing the point $w$, it follows that $w \in \operatorname{ConvexHull}(\{x_1, \dots, x_{16}\})$.
\end{proof}

This lemma leads to a formal way to define the set of boundary points generated by the directions. For each $d_i \in D$ we define the intersection point $x_i$ with the boundary of $L^H$:

\begin{definition}
Let $\lambda_i = \sup \{ \lambda \ge 0 : w + \lambda d_i \in L^H \}$.
We define $x_i = w + \lambda_i d_i$. Thus we obtain a set of points $X = \{x_1,\dots,x_{16}\} \subset \partial L^H$.
\end{definition}

 Now we construct eight subsets of $L^H$ corresponding to the facets of the hypercube. For each $i=1, 2, \dots, 8$ define the polyhedral cone $K_i = \operatorname{conv} \{ \alpha x_j : j \in F_i, \; \alpha \in \mathbb{R}_{>0} \}$.

Define $C_i = L^H \cap \left( \bigcap_{M \in \mathcal{M}_i} M \right)$. The set $C_i$ is an intersection of convex sets containing $w$, hence it is convex and contains $w$.

\begin{statement}
The sets $C_1,\dots,C_8$ form a partition of $L^H$.
\end{statement}

\begin{proof}
Observe that $\bigcup_{1 \leq i \leq 8} K_i = \mathbb{R}^4$. Thus,

\[
    \bigcup_{1 \leq i \leq 8} C_i = \bigcup_{1 \leq i \leq 8} (K_i \cap L^H) = \left(\bigcup_{1 \leq i \leq 8} K_i\right) \cap L^H = L^H.
\]

\end{proof}

\subsection{Applying a random rotation}

The combinatorial structure described above is fixed. However, the geometry of the partition depends on the  rotation of the direction set $D$.

For the initialization of the directions we will use a random orthogonal matrix. The following statement means that we choose a distribution that is ``uniform'' (left--invariant) on the group of orthogonal transformations of the space.

\begin{lemma}[see \cite{mitchell2008sampling}]
Let $P \in \mathbb{R}^{n \times n}$ be a random matrix with elements sampled from the standard Gaussian distribution. Let $P = QR$, where $Q, R \in \mathbb{R}^{n \times n}$ are the QR-decomposition of this matrix, and $Q$ is a (random) orthogonal matrix. Then the distribution of the generated matrix is the Haar distribution. 
\end{lemma}

\subsection{Lower and upper estimates}

 The main objective function can therefore be written as

\begin{equation}
F(w,d'_1, \dots, d'_{16}) = \max_{i=1,\dots,8} \operatorname{diam} C_i(w,d'_1, \dots, d'_{16}) \xrightarrow{} \min. \tag{$\mathcal{P}$}
\end{equation}

Exact diameter computation inside $L^H$ is technically complicated, because we consider the partitions in $\mathbb{R}^4$. It is known that the distance between two unit circles cannot be expressed in radicals~\cite{neff1990finding}. The following scheme has been applied to obtain a lower bound for the objective. Let us define a finite point approximation for  $C_i$, $1 \le i \le 8$:

\begin{enumerate}
    \item Let $m$ be a positive integer parameter used to construct a lower-bound approximation based on a uniform discretization of the open interval $(0, 1)$ with grid points given by $\theta_r = \frac{r}{m+1}$ for $r = 1, \ldots,  m$.

    \item Next, for every two directions $d'_p, d'_q$, with $p, q \in F_i$, forming $C_i$,  we consider the projection onto the cover boundary induced by the aforementioned grid. Specifically, we define
    $$y_{pqr} = \sup \{ \lambda \ge 0 : w + \lambda (\theta_r d'_p + (1-\theta_r)d'_q) \in L^H \}.$$

    \item Finally, we construct the set $$\tilde C_i = \{w\} \cup \{d'_p : p \in F_i\} \cup \{y_{pqr} : p,q \in F_i; \, 1 \leq r \leq m\}.$$ Notice that $\tilde C_i \subset C_i$, so $\operatorname{diam} \tilde C_i \le \operatorname{diam} C_i$. Figure~\ref{fig:lower_bound_construction} clarifies this construction in the planar case.
\end{enumerate}

\begin{figure}[htbp]
\centering

\begin{subfigure}[t]{0.48\textwidth}
    \centering
    \includegraphics[width=\textwidth]{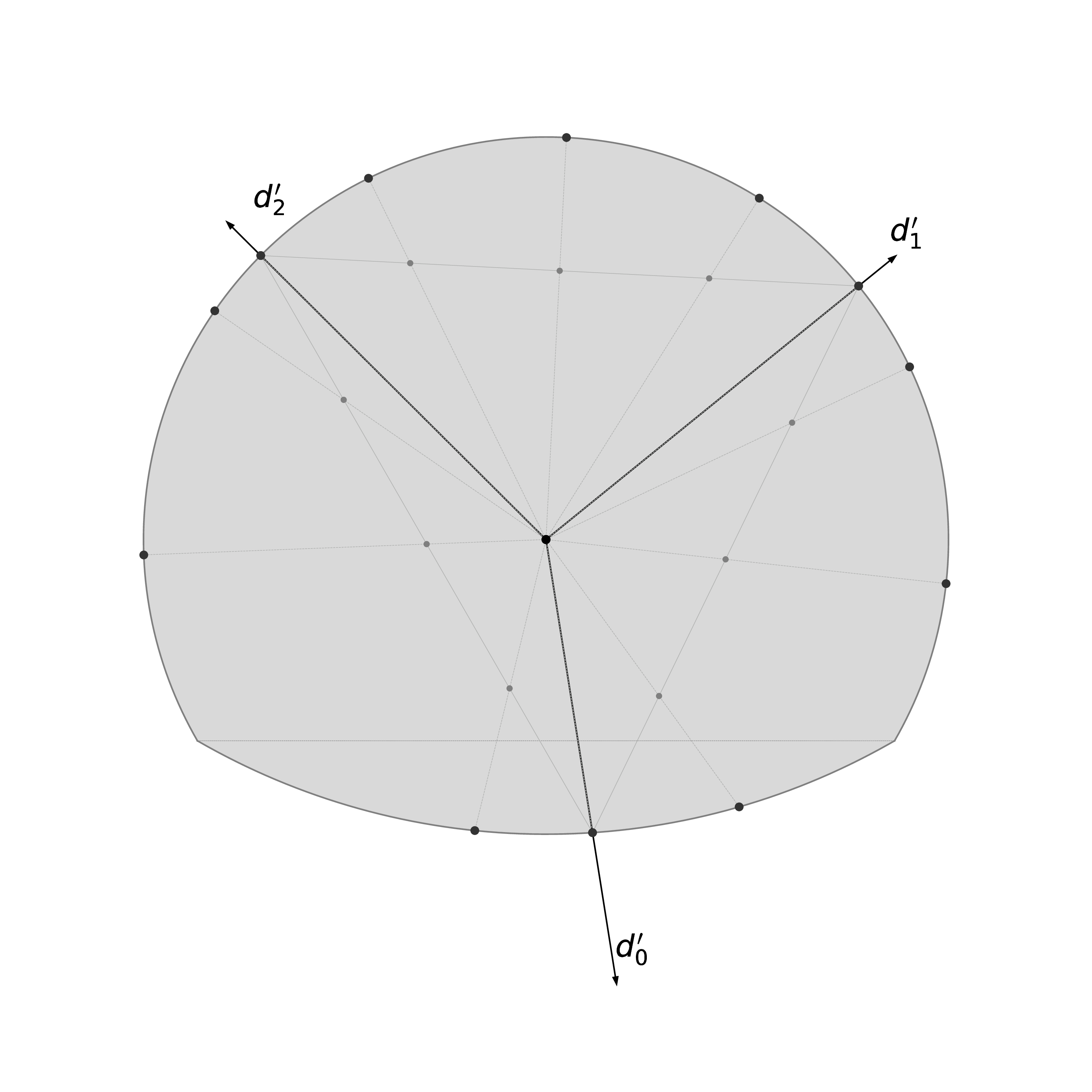} 
    \caption{Lassak cover $L^H_{(2)}$, \\ $H = \emptyset, m = 3$ }
    \label{fig:2d_lassak_cover_lower_bound}
\end{subfigure}
\hfill
\begin{subfigure}[t]{0.48\textwidth}
    \centering
    \includegraphics[width=\textwidth]{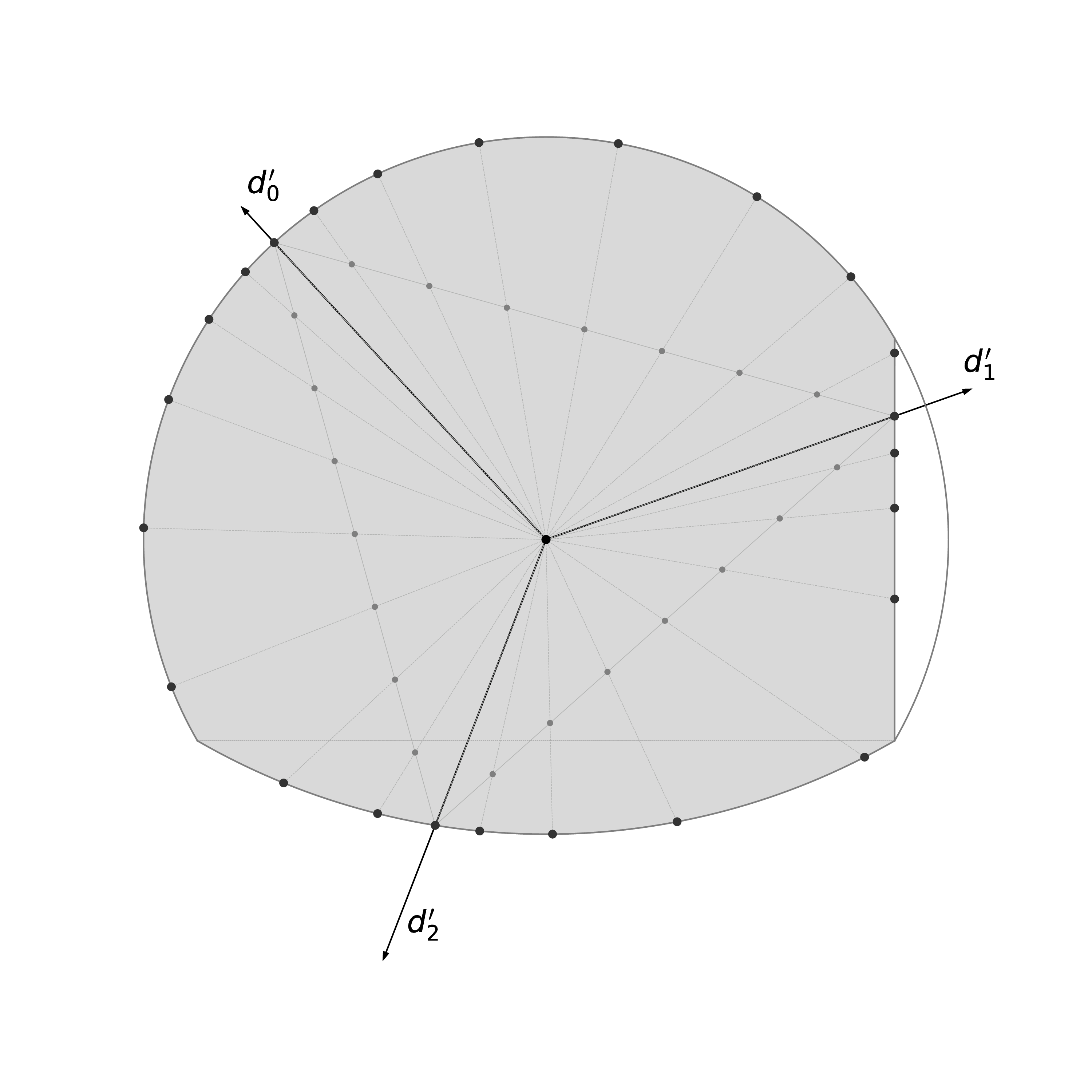} 
    
    \caption{Truncated Lassak cover $L^H_{(2)}$, \\ $|H| = 1, m = 7$}
    \label{fig:2d_truncated_lassak_cover_lower_bound}
\end{subfigure}

\caption{Lower bound, the Lassak cover (~\ref{fig:2d_lassak_cover_lower_bound}) and its truncated version (~\ref{fig:2d_truncated_lassak_cover_lower_bound}), division into 3 parts in the planar case.}
\label{fig:lower_bound_construction}

\end{figure}

 The convex linear combinations, the intersection points and the matrix of pairwise distances can be efficiently computed using  matrix and tensor computations. This allows us to use the resources of multi-core CPUs and GPUs without explicit parallel programming.
 \begin{equation}
     \tilde F(w,d'_1, \dots, d'_{16}) = \max_{i=1,\dots,8} \operatorname{diam} \tilde C_i(w, d'_1, \dots, d'_{16}) \xrightarrow{} \min. \tag{$\mathcal{P}_{lower}$}
 \end{equation}
 
 The problem $\mathcal{P}_{lower}$ was used repeatedly to find promising partitions using stochastic gradient descent, which were then optimized using more computationally intensive methods.

In order to justify the correctness of the partition in the Borsuk problem, we need either an approximation error estimate or an upper bound. In the present paper, we will restrict ourselves to the ``ancient Greek'' upper bound, in which the partition is applied to some circumscribed polyhedron for the Lassak cover. More precisely, we must define a series of circumscribed polyhedra that can approximate the Lassak cover arbitrarily closely.
  \begin{equation}
     \overline{F}(w,d'_1, \dots, d'_{16}) = \max_{i=1,\dots,8} \operatorname{diam} \overline{C}_i(w, d'_1, \dots, d'_{16}) \xrightarrow{} \min. \tag{$\mathcal{P}_{upper}$}
 \end{equation}

 Here $\overline{C}_i$ is the part of the circumscribed polyhedron (see Appendix~\ref{app:polyhedron_approximation} for more details). The problem $\mathcal{P}_{upper}$ was used with gradient-free optimization methods in order to ``fine-tune'' the partition. 

\subsection{Multiple runs}

The following algorithm clarifies the process described in this section. Recall that $\mathcal{S}$ is a universal covering system containing several truncated Lassak covers and the enclosing polyhedron $P_i$ for each $C_i \in \mathcal{S}$. The construction of this polyhedron for diameter evaluation is described in the Appendix. After that we run Algorithm~\ref{alg:lassak_diameter} for the elements of $\mathcal{S}$ and ensure that $d_{\text{best}} < 1$ for each $L^H$.

\begin{algorithm}
\caption{Search for a partition of the cover}
\label{alg:lassak_diameter}
\KwIn{Truncated Lassak cover $L^H$, set of directions $D$, facet structure $\mathcal{F}$, number of restarts $n_{\text{restarts}}$}

\KwOut{Best diameter $d_{\text{best}}$}

Initialize best diameter: $d_{\text{best}} = 1$;

\For{$i \leftarrow 1$ \KwTo $n_{\text{restarts}}$}{
Generate a random matrix $P \in \mathbb{R}^{4 \times 4}$ with $P_{ij} \sim \mathcal{N}(0, 1)$;

Apply QR decomposition: $P = QR$, where $Q \in O(4)$ is a random orthogonal matrix;

Replace each direction $d_i \in D$ by $d'_i = Q d_i$;
$w = c_0$;

Construct the partition of $L^H$ into sets $C_1, \dots, C_8$;

Find a local minimum of the optimization problem $\mathcal{P}_{lower}$ using stochastic gradient descent. Let $(\tilde w, \tilde D)$ be the local minimum of $\tilde F$;

If $\tilde F(\tilde w, \tilde D) < 1$, then apply "fine-tuning", i.e., gradient-free optimization to problem $\mathcal{P}_{upper}$, using $(\tilde w, \tilde D)$ as the starting point. Let $d_{\text{upper}}$ be the final value of $\overline{F}$;

\If{$d_{\text{upper}} < d_{\text{best}}$}{
$d_{\text{best}} = d_{\text{upper}}$;
}
}
\Return $d_{\text{best}}$;
\end{algorithm}

\section{Results}

Table~\ref{partition_results_table} summarizes the experimental results for partitioning elements of the above-mentioned universal covering system into 8 parts. 

To search for a partition, 1–2 thousand multistart runs were performed for each UCS element. Further computations were carried out using the \textit{scipy} and \textit{qhull} libraries, as well as an implementation of the Adam algorithm~\cite{Adam2017} in the PyTorch~\cite{Pytorch2019} library. For local optimization, the Nelder–Mead algorithm~\cite{nelder1965simplex} and the black-box optimization algorithm proposed by P.~Knysh and Y.~Korkolis~\cite{knysh2016blackbox} were also applied. Switching between algorithms was sometimes done manually, and the entire partition search process was not fully formalized. The total computation time was approximately 10 CPU days using a workstation with an EPYC 7532 CPU (32 cores).

Detailed specifications of the UCS elements obtained from this experiment, together with the corresponding validation scripts, can be found in the repository~\cite{our_project_repo}.


\begin{center}
\begin{table}
\caption{Paritions into 8 parts.}
\label{partition_results_table}
\centering
\setlength{\tabcolsep}{4pt} 
\renewcommand{\arraystretch}{1.1} 
\begin{tabular}{|c|c|c|c|}
\hline
$i$ & \shortstack{UCS element, Proposition~\ref{st:ucsR4}} & diameters & $d_{\text{max}}$ \\
\hline \hline
1 & $L_{(4)} \cap H_1 \cap H_2 \cap H_3^+ \cap H_4^+ \cap H_5^+ \cap H_6^+$ 
& {\scriptsize \shortstack{ \{0.99775 0.99763 0.99906 0.99683 \\ 0.99670 0.99762 0.99374 0.99339\} }} 
& 0.99906 \\
\hline
2 & $L_{(4)} \cap H_1 \cap H_2 \cap H_3^+ \cap H_4^+ \cap H_5^+ \cap H_6^-$ 
& {\scriptsize \shortstack{ \{0.99735, 0.99809, 0.99549, 0.99733, \\ 0.99736, 0.99503, 0.99615, 0.99566\} }} 
& 0.99809 \\
\hline
3 & $L_{(4)} \cap H_1 \cap H_2 \cap H_3^+ \cap H_4^+ \cap H_5^- \cap H_6^-$ 
& {\scriptsize \shortstack{ \{0.99862, 0.99819, 0.99729, 0.99987,\\ 0.99961, 0.99645, 0.99734, 0.99696\} }} 
& \textbf{0.99987} \\
\hline
4 & $L_{(4)} \cap H_1 \cap H_2 \cap H_3^+ \cap H_4^- \cap H_5^+ \cap H_6^- \cap W^+$ 
& {\scriptsize \shortstack{ \{0.99978, 0.99845, 0.99821, 0.99846,\\ 0.99800, 0.99634, 0.98538, 0.99037\} }} 
& 0.99978 \\
\hline
\end{tabular}
\end{table}
\end{center}

\section{Conclusion}

The presented estimates required a substantial amount of computation, as well as preliminary work not included in this paper. We used a relatively simple method for choosing the initial approximation. Based on our experience, it can be assumed that if further improvement of the Borsuk number estimate in dimension 4 is feasible using these means, it would require either a significantly different approach or much greater computational resources. It seems unlikely that the Borsuk conjecture in dimension 4 can be proved using the methods we employed.

\section*{Disclosure statement}

The authors report there are no competing interests to declare.

\section*{Data availability statement}

Descriptions of the partitions and the corresponding verification and visualization code are available in our repository~\cite{our_project_repo}.

\bibliographystyle{abbrv}
\bibliography{bibliography}

\appendix

\section{Construction of a circumscribed polyhedron in \texorpdfstring{$\mathbb{R}^n$}{Rn}}\label{app:polyhedron_approximation}

We describe a procedure for constructing a polyhedral outer approximation (circumscribed polyhedron) of a set arising from a universal covering set (UCS) in $\mathbb{R}^n$.

\begin{enumerate}

\item Let $L^H = L_{(n)} \cap \left( \bigcap_{i=1}^{t} \{x \in \mathbb{R}^n : \langle x, a_i \rangle + b_i \le 0\} \right)$ be a subset of the UCS obtained as the intersection of the Lassak cover $L_{(n)}$ (see Theorem~\ref{thm:lassak}) with a finite family of $t$ halfspaces defined by
\[
H = \{(a_i, b_i) : a_i \in \mathbb{R}^n,\ b_i \in \mathbb{R},\ i=1,\dots,t\}.
\]

\item Consider a uniform discretization of the boundary of the unit hypercube:
\[
G = \Bigl\{0, \tfrac{1}{m}, \dots, \tfrac{m-1}{m}, 1\Bigr\}^n \setminus \Bigl\{\tfrac{1}{m}, \dots, \tfrac{m-1}{m}\Bigr\}^n,
\]
i.e., $G = \left\{ (i_1,\dots,i_n) : i_j \in \{0,\tfrac{1}{m},\dots, \tfrac{m-1}{m}, 1\},\ \exists j \text{ such that } i_j \in \{0,1\} \right\}.$ This set represents the discretization of the hypercube boundary. The integer parameter $m$ used as ``$\text{grid\_size}$'' in our code~\cite{our_project_repo}. Table~\ref{partition_results_table} presents the results for a grid size of $m = 17$.

\item Let the ball $B_0$, forming $L_{(n)}$, is centered at $c_0$ with radius $r$. Define
\[
G_1 = \{\, c_0 + r \cdot \left(p - (1/2,\dots,1/2)\right) : p \in G \,\},
\]
which corresponds to the discretized boundary of a hypercube centered at $c_0$.

For each $X \in G_1$, consider the ray $c_0X$ and define
\[
G_2 = \{\, c_0X \cap \partial L_{(n)} : X \in G_1 \,\},
\]
i.e., the set of intersection points of these rays with the boundary of $L_{(n)}$.

\item For each point $p \in G_2$, let $h_p$ be the supporting (tangent) hyperplane to $L_{(n)}$ at $p$ (in practice, to one of the balls defining $L_{(n)}$). Define the polyhedron $P = \bigcap_{p \in G_2} H_p$,
where $H_p$ is the halfspace bounded by $h_p$ and containing $c_0$.

\item The set $P$ is a convex polyhedral outer approximation of $L_{(n)}$. Define $P_H = P \cap \left( \bigcap_{i=1}^{t} \{x : \langle x, a_i \rangle + b_i \le 0\} \right),$ which serves as a polyhedral approximation of $L^H$.

\end{enumerate}

Figure~\ref{fig:planar_case_construction} illustrates this construction in the planar case. For $n = 2$ the intersection of the two spheres forming the boundary of the Lassak cover is a circle of radius $\frac{1}{2}$. This figure shows a discretized set of directions, the corresponding boundary points obtained by radial projection from the center, and the supporting lines defining a polygonal outer approximation of these covers in planar case. Note that, due to the simmetricity of the following truncation (see Figure~\ref{fig:2d_truncated_lassak_cover}), both filled sets are universal covering sets in $\mathbb{R}^2$.

\begin{figure}[htbp]
\centering

\begin{subfigure}[t]{0.46\textwidth}
    \centering
     \includegraphics[width=\textwidth]{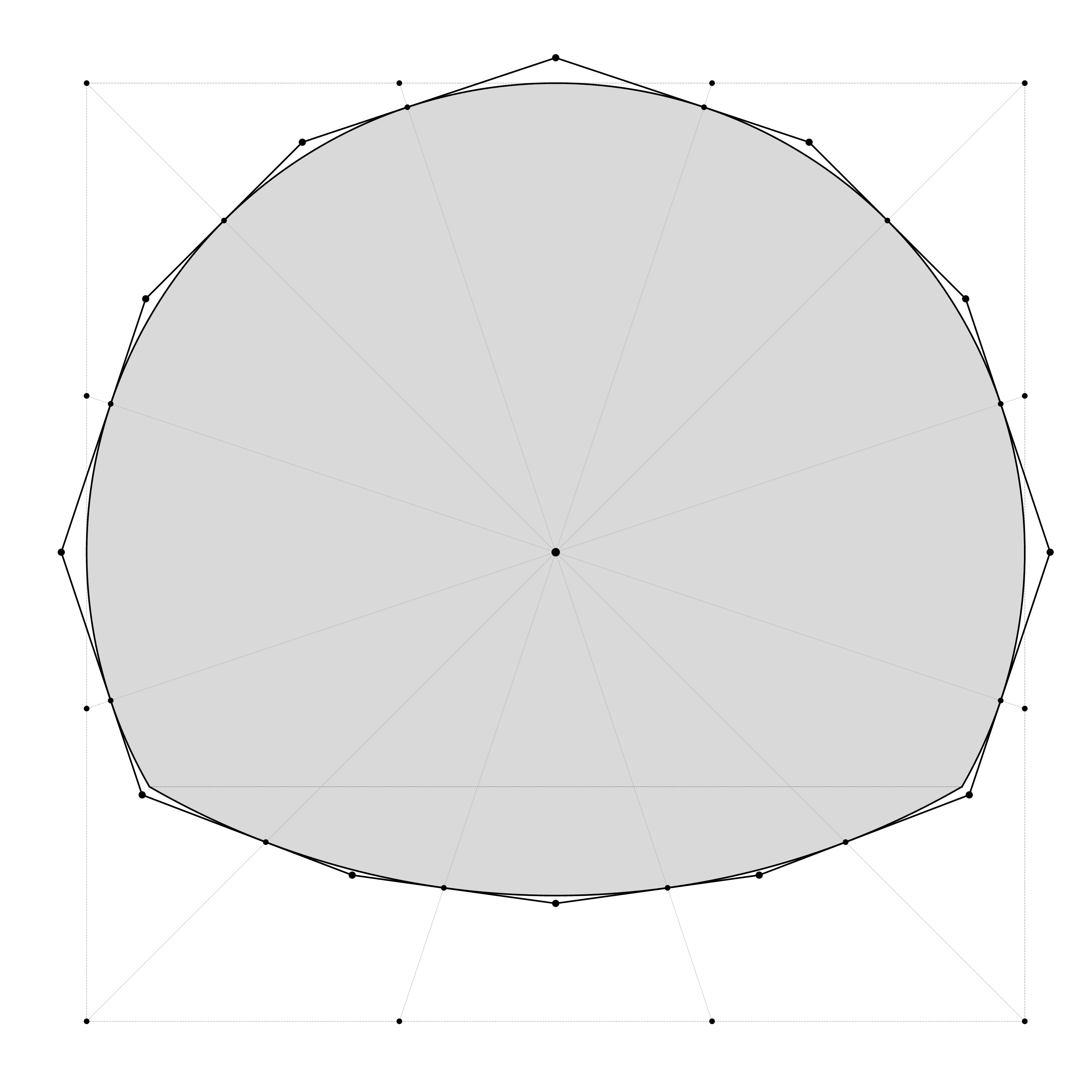} 
     
    \caption{Lassak cover $L^H_{(2)}$, $H = \emptyset$ }
    \label{fig:2d_lassak_cover}
\end{subfigure}
\hfill
\begin{subfigure}[t]{0.46\textwidth}
    \centering
    \includegraphics[width=\textwidth]{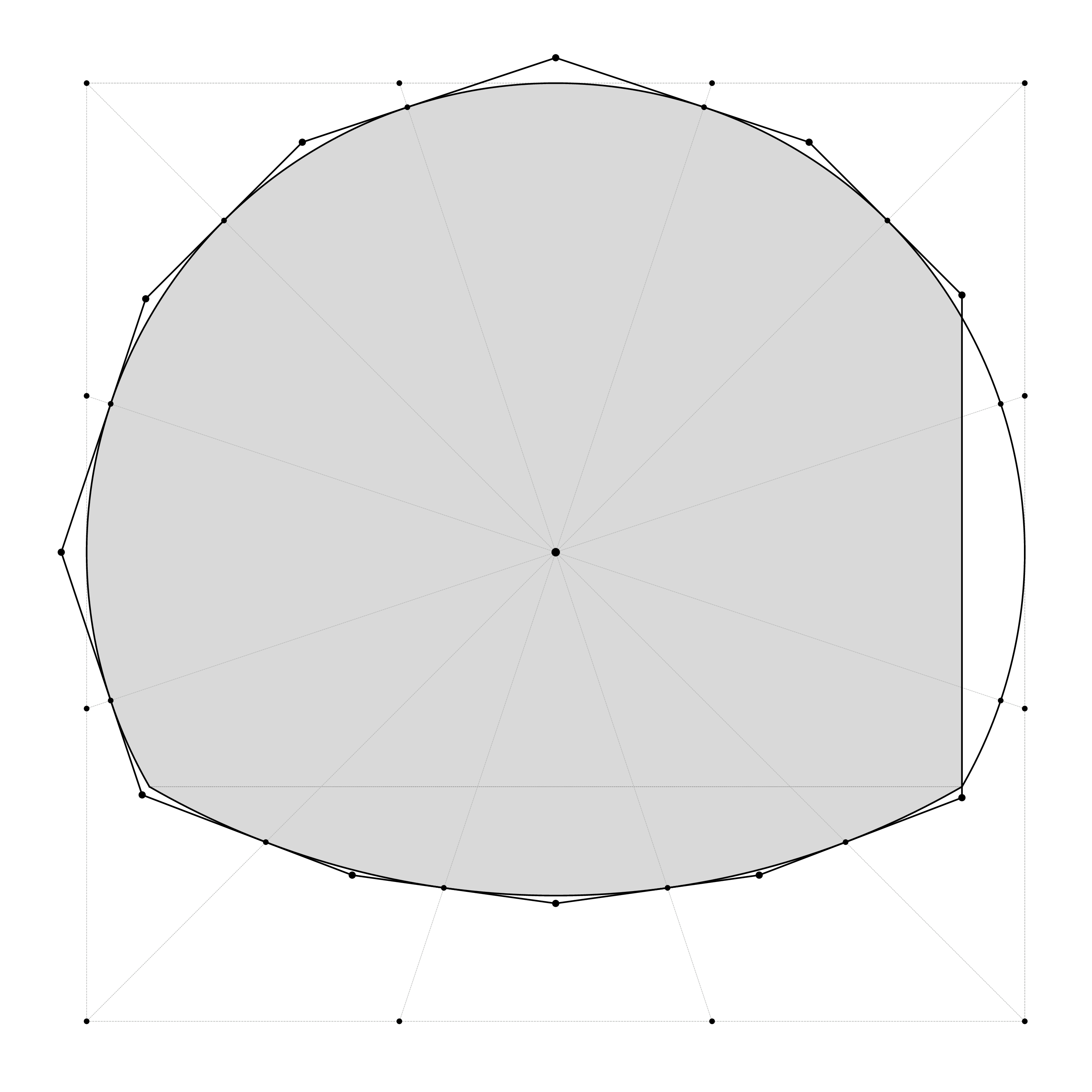} 
    
    \caption{Truncated Lassak cover $L^H_{(2)}$, $|H| = 1$}
    \label{fig:2d_truncated_lassak_cover}
\end{subfigure}

\caption{Circumscribed polyhedron approximation of Lassak cover (~\ref{fig:2d_lassak_cover}) and its truncated version (~\ref{fig:2d_truncated_lassak_cover}) in planar case.}
\label{fig:planar_case_construction}

\end{figure}

\end{document}